\documentclass[12pt,a4paper]{article}
\usepackage{graphics}
\usepackage[dvipdfm]{graphicx}
\usepackage{color}
\usepackage{natbib}
\usepackage{amsthm, amsmath,amssymb}
\usepackage{enumerate}

\usepackage[T1]{fontenc}
\newtheorem{teo}{Theorem}[section]

\newtheorem{prop}{Proposition}[section]
\newtheorem{cor}{Corollary}[section]

\newcommand{\appsection}[1]{\let\oldthesection\thesection
\renewcommand{\thesection}{Appendix \oldthesection}
\section{#1}\let\thesection\oldthesection}

\newcommand{\real}{\mathbb{R}}

\newcommand{\var}{\mathrm{Var}}
\newcommand{\te}{\theta}
\newcommand{\na}{\nabla}

\title{Reference Bayesian analysis for hierarchical models}

\author{Tha\'is C. O. Fonseca \thanks
{Department of Statistics, Universidade Federal do Rio de Janeiro, RJ, Brazil,  thais@im.ufrj.br} \and Helio S. Migon \thanks{Department of Statistics, Universidade Federal do Rio de Janeiro, RJ, Brazil,  migon@im.ufrj.br} \and Heudson Mirandola \thanks{Department of Mathematics, Universidade Federal do Rio de Janeiro, RJ, Brazil,  mirandola@im.ufrj.br}}

\date{}

\begin{document}
\maketitle
\abstract
This paper proposes an alternative approach for constructing invariant Jeffreys prior distributions tailored for hierarchical or multilevel models. In particular, our proposal is based on a flexible decomposition of the Fisher information for hierarchical models which overcomes the marginalization step of the likelihood of model parameters. The Fisher information matrix for the hierarchical model is derived from the Hessian of the Kullback-Liebler (KL) divergence for the model in a neighborhood of the parameter value of interest. Properties of the KL divergence are used to prove the proposed decomposition. Our proposal takes advantage of the hierarchy and leads to an alternative way of computing Jeffreys priors for the hyperparameters and an upper bound for the prior information. While the Jeffreys prior gives the minimum information about parameters, the proposed bound gives an upper limit for the information put in any prior distribution. A prior with information above that limit may be considered too informative.
From a practical point of view, the proposed prior may be evaluated computationally as part of a MCMC algorithm. This property might be essential for modeling setups with many levels in which analytic marginalization is not feasible.
We illustrate the usefulness of our proposal with examples in mixture models, in model selection priors such as lasso and in the Student-t model. 


\section{Introduction}
There has been an increasing interest over the last decade in the construction of flexible and computational efficient models such as neural networks, clustering, multilevel, spatial random effects and mixture models. These models are setup in an hierarchical fashion which naturally describe several real data characteristics. The hierarchy allows for flexibility and complexity is take into account by adding extra levels in the model. In this context, subjective prior elicitation for the hyperparameters is not always trivial as parameters are very often defined in low levels of the hierarchy and lack practical interpretation. Frequently, the parameters are calibrated or estimated using empirical Bayes approaches. Recent work has focus attention in the specification of hyperpriors for parameters in low levels of the hierarchy such as the penalizing complexity priors of \cite{Rue17} which is weakly informative and penalizes parameter values far from the base model specification. In the context of full Bayesian analysis, objective prior specifications may be considered instead of calibration, empirical Bayes or weakly informative priors. 

Consider $y$ the vector of observed responses and a probabilistic parametric model $\mathcal{M}=\{f(\,\cdot\,\mid \theta)\ ;\ \theta\in \Theta\}$. Direct inference about $\theta$ using the integrated likelihood $f$ is not always trivial and levels of hierarchy are often introduced in the modelling to allow for feasible inferences regarding $\theta$ due to conditional independence given the latent variables. As follows, we assume that the mechanism which has generated the data is a data augmentation process: given $\theta$, a value of $w$ is selected from $f_2(w\mid \theta)$ and, given $w$, a value of $y$ is selected from $f_1(y\mid \theta,w)$. The model for $w\in \real^d$ depends on the application and is often chosen to allow easier inferences about $\theta$ or due to easier interpretability. In this setup the model has two elements: the extended likelihood $f_1(y\mid w,\theta)$ and the model for the latent variable $f_2(w\mid \theta)$. The marginal likelihood for $\theta$ given the data is obtained by integration
\begin{equation}\label{eqmix}
f(y\mid \theta)=\int f_1(y\mid \theta,w)d F_2(w\mid \theta).
\end{equation}
with $F_2$ the cumulative distribution associated with $f_2$. Notice, however, that the integration step is not usually desirable as in general flexible hierarchical modelling interest lies in making inference for the hidden effects $w$ as well as the parameters $\theta$. The introduction of latent variables in the inferential problem brings great benefits as very often the complete conditional distributions of $w\mid \theta,y$ have nice explicit forms. In an iterative algorithm as the one proposed by \cite{Tan87}, a sample is obtained from $p(w\mid \theta , y)$ and a sample from $p(\theta\mid w,y)$ is obtained conditional on the sampled value of $w$. Both complete conditional distributions are often easy to sample from.

Another important aspect of the hierarchical approach is that the model is usually a flexible version of a base model as discussed in \cite{Rue17}. This naturally leads to ill behaved likelihoods as the lower levels in the hierarchy may converge to a constant and the model converges to the base model. This may happen with quite high probability if the sample size is not too large. This bad behaviour of the likelihood function is not corrected by reparametrization and the use of informative priors will mean that the inference made a priori and a posteriori will be approximately the same. Examples with this characteristic are the Student-t model \citep{FernSte99,Fons08}, the skewed normal distribution \citep{Liseo06}, mixture models \citep{Gir88,Robert18}, the hyperbolic model \citep{Fons12}. In these cases, the use of Jeffreys priors correct the ill behaved likelihood and provide better frequentist properties than the ones achieved by using informative priors.

The pioneer attempts to provide default inferences were based on setting uniform prior for unknown parameters \citep{Bayes1763}.
However, these proposals are not invariant to transformations. \cite{Jef46} introduced invariant priors based on the Riemannian geometry of the statistical model. In particular, the divergence considered for prior derivation were the squared Hellinger and the Kullback-Leibler distances. Both divergences behave locally as the Fisher information. More recently, \cite{Geo93} considered a large class of invariant priors based on discrepancy measures. The motivation of these authors was to think in terms of the parametric family $\mathcal{M}$ instead of the parameter domain $\Theta$. Thus, the prior weight assigned to a neighbor value of $\theta$, say $\theta^*$, depends directly on the discrepancy between the parametric family members $f(\cdot\mid \theta)$ and $f(\cdot\mid \theta^*)$.

A way to measure information in an inferential problem is through discrepancy functions such as Fisher information. In the context of hierarchical models, it is natural that the Fisher information about $\theta$ obtained from the complete data problem $(y,w)$ must be larger than the one obtained from the incomplete data problem $y$. Indeed, if $w$ and $y$ were both known it would be easier to make inference about $\theta$. 

In this work, we present a Fisher information decomposition for a general hierarchical problem which allows for computing the information about $\theta$ in the incomplete model $f$ and in the extended model $f_1$ and $f_2$ overcoming the marginalization step. From the proposed decomposition we obtain an alternative way to compute the Jeffreys prior for $\theta$ and an upper bound for this prior. The Jeffreys rule prior gives the minimum prior information for inference about $\theta$. On the other hand, any prior which gives more information than the upper bound may be too informative. 

Section 2 proposes the Fisher information decomposition, the upper bound for the Jeffreys prior of $\theta$ and an alternative way to compute the usual Jeffreys rule prior for an hierarchical model. Some hierarchical models, often  explored in the current literature are presented in the following sections. In section 3 the discrete mixture model \citep{Gir88} is presented and the Jeffreys prior is obtained without the usual marginalization step. An alternative proof for integrability of the resulting Jeffreys prior is also presented. Section 4 discusses the Student-t model with unknown degrees of freedom. The model is written as a Gaussian mixture model and the Jeffreys rule prior is obtained directly for the hierarchical model. Section 5 presents Jeffreys prior for Lasso parameter in a regression model also overcoming the marginalization step.

\section{Fisher matrix decomposition and proposed reference prior}

We recall that the Fisher information matrix for a probabilistic parametric model $\mathcal{M}=\{f(\,\cdot\,\mid \theta)\ ;\ \theta\in \Theta\}$ is defined by the expected matrix of the observed information, 
\begin{equation}\label{fishermatrix}
I_y(\theta)_{ij}=-E_{f(y\mid \theta)}\left[\frac{\partial^2}{\partial \theta_i\theta_j}\log f(y\mid \theta)\right ].
\end{equation}
It can also be seen as the variance of the score function, i.e., $I(\te) = \var_y[\na_\te \log f(y\mid \te)]$. 

In many situations, analytical or even numerical expressions for the Fisher information matrix is prohibitive. Moreover, the use of Monte Carlo methods for approximate \eqref{fishermatrix} can yield high variance. So, in order to reduce variance, we can use the Rao-Blackwell theorem, by carrying out analytical computation as much as possible. Based on this, we propose a Rao-Blackwell-type theorem to be used as an alternative method to compute the Fisher information matrix on hierarchical structures of the model (\ref{eqmix}). For the next result, we consider
\begin{enumerate}[(i)]
\item $I_y(\theta\mid w)$ the Fisher information matrix of the extended likelihood $f_1(y\mid w,\theta)$;
\item $I_w(\te)$ the Fisher information matrix of the latent variable $f_2(w\mid \theta)$;
\item $I_{y,w}(\te)$ the Fisher information matrix of the complete model $f(y,w\mid \te)$;
\item $I_{w}(\te\mid y)$ the Fisher information matrix of the complete conditional $f(w\mid \te,y)$.
\end{enumerate}
We obtain the following
\begin{teo}\label{FIM-hierarchical} The Fisher information matrices of the hierarchical model \eqref{eqmix} can be decomposed as 
\begin{equation}\label{FIM-decomposition}
I_{y,w}(\theta) = I_y(\theta) + E_{y}[I_w(\theta\mid y)] = I_w(\theta) + E_{w}[I_y(\theta\mid w)].
\end{equation}
\end{teo}

For the sake of simplicity, in the expression above, $E_{y}$ denotes $E_{f(y\mid \theta)}$ and $E_{w}$ denotes $E_{f_2(w\mid \theta)}$. If $E_{y}[I_w(\theta\mid y)]$ is not easily computed then we may take advantage of the hierarchy, which leads to $E_{f(y\mid \theta)}[I_w(\theta\mid y)]=E_{f_2(w\mid \theta)}\{E_{f_1(y\mid w,\theta)}[I_w(\theta\mid y)]\}$.\\

To prove Theorem \ref{FIM-hierarchical}, we observe that the Fisher information matrix can be derived from the Kulback-Leibler (KL) divergence. Namely, for any parametric model $p(y\mid \theta)$ it holds 
\begin{equation}\label{expansao_KL}
KL[p(y\mid \theta):p(y\mid \theta+d\theta)] = \frac{1}{2}\,d\theta^T\, I_y(\theta)\, d\theta + O(|d\theta|^3).    
\end{equation}
In other words, the Fisher information matrix $I(\theta) = [I(\theta)_{ij}]$ can be obtained as the Hessian of the KL divergence
\begin{equation}
I(\theta)_{ij} = \frac{\partial^2}{\partial\theta^i\partial\theta^j}KL[p(y\mid \theta) : p(y\mid \theta')]\mid_{\theta'=\theta}. 
\end{equation}

The proposition below is the main ingredient to derive \eqref{FIM-decomposition}. 

\begin{prop}\label{Kullback-leibler_div} Consider the probability distributions $f(y,w)$ and $g(y,w)$, where $y,w$ are random vectors. The Kullback-Leibler divergence satisfy:
\begin{align}\label{KLjoint}
KL[f(y,w) : g(y,w)] &= KL[f(w): g(w)] + E_{f(w)}[KL[f(y\mid w):g(y\mid w)]\nonumber \\ &= KL[f(y):g(y)] + E_{f(y)}[KL[f(w\mid y): g(w\mid y)].
\end{align}
\end{prop}

Proposition \ref{Kullback-leibler_div} follows by a direct computation. In fact, the Kullback-Leibler divergence of $f(y,w)=f(y\mid w)f(w)$ and $g(y,w)=g(y\mid w)g(w)$ is given by
\begin{align}
KL[f(y,w):&\ g(y,w)] = \int_{\mathcal Y\times \mathcal W} f(y,w)\log(f(y,w)/g(y,w))dydw \nonumber\\
& = \int_{\mathcal W}f(w)\int_{\mathcal Y} f(y\mid w) (\log(f(y\mid w)/g(y\mid w))+\log(f(w)/g(w))\nonumber\\
&= \int_{\mathcal W}f(w)\log(f(w)/g(w)) + \int_{\mathcal W}f(w) \int_{\mathcal Y} f(y\mid w)\log(f(y\mid w)/g(y\mid w))\nonumber\\
&= KL[f(w):g(w)] + E_{f(w)}[KL[f(y\mid w):g(y\mid w)] \label{kl_decomp}
\end{align}
The first term of the third equality holds since $\int_{\mathcal Z}f(z\mid w)dz = 1$. The second equality in \eqref{KLjoint} holds similarly. \qed

Now, we will prove Theorem \ref{FIM-hierarchical}. For the hierarchical model \eqref{eqmix}, we consider the likelihood $f(y,w\mid \theta)$ of $\theta$ given the joint random vector $(y,w)$ and the complete conditional posterior distribution $f(w\mid y,\theta)$. Applying Proposition \ref{Kullback-leibler_div}, we obtain
\begin{align*}
KL[f(y,w\mid \theta):f(y,w \mid \theta')] &= KL[f_2(w\mid \theta):f_2(w\mid \theta')] + E_{w\mid\theta}[KL[f_1(y\mid w,\theta):f_1(y\mid w,\theta')]\\
&= KL[f(y\mid \theta):f(y\mid \theta')] + E_{y\mid \theta}[KL[f(w\mid y,\theta):f(w\mid y,\theta')].
\end{align*}
Taking the Hessian of the above equalities at $\theta'=\theta$, we obtain \eqref{FIM-decomposition}. \qed

We can use Theorem \ref{FIM-hierarchical} to compute the Fisher information matrices of high-level hierarchical models:
\begin{equation}\label{high-level_hier}
\begin{array}{l}
y\sim f(y\mid w_k,\ldots,w_1,\theta),\\ w_k\sim f_k(w_k\mid w_{k-1},\ldots,w_1,\theta),\\
\ldots, \\
w_1\sim f_1(w_1\mid \theta).
\end{array}    
\end{equation}

We write $w=(w_1,\ldots,w_k)$ and $\mathbf{w_t} = (w_1,\ldots,w_t)$, for $t=1,\ldots,k$. Applying recursively Theorem \ref{FIM-hierarchical}, it holds
\begin{align*}
I_{y,w}(\theta) &= I_y(\theta) + E_y[I_{w_1}(\theta\mid y)] + E_{y,\mathbf{w_1}}[I_{w_2}(\theta\mid y,\mathbf{w_1})]+\ldots + E_{y,\mathbf{w_{k-1}}}[I_{w_k}(\theta\mid y,\mathbf{w_{k-1}})] \\
&= I_{w_1}(\theta) + E_{\mathbf{w_1}}[I_{w_2}(\theta\mid \mathbf{w_1})] + \ldots + E_{\mathbf{w_{k-1}}}[I_{w_k}(\theta\mid \mathbf{w_{k-1}})] + E_{w}[I_y(\theta\mid w)].
\end{align*} 

As a particular case, consider the hierarchy model: 
\begin{equation}\label{hierarchi_particular}
y\mid w \sim f_1(y\mid w) \ \mbox{ and } \  w\mid \theta \sim f_2(w\mid \theta).
\end{equation}
Since $f_1(y\mid w)$ does not depend on $\theta$, it holds that $I_y(\theta\mid w) = 0$. Applying Theorem \ref{FIM-hierarchical} we obtain $I_{y,w}(\theta) = I_w(\theta)+E_w[I_y(\theta\mid w)] = I_w(\theta)$. Thus, it holds
\begin{cor}\label{FIM-hierarchical_cor} For the hierarchical model is $y\mid w \sim f_1(y\mid w)$ and $w\mid \theta\sim f_2(w\mid \theta)$, the posterior Fisher information matrices satisfy 
\begin{equation}\label{posteriormean}
I_y(\theta) = I_w(\theta) - E_y[I_w(\theta\mid y)].
\end{equation}
In particular, if we assume the Jeffreys' prior $\pi_w(\theta)$ of the model $w\mid \theta$ is proper then the Jeffreys' prior $\pi_y(\theta)$ of the marginal likelihood distribution $y\mid \theta$ is also proper.
\end{cor}

It is worth to emphasize that both par\^ameters $w$ and $\theta$ can be multivariate. In this case, we can also use of Theorem \ref{FIM-hierarchical} to give a lower bound of the Jeffreys rule prior $\pi_{y}(\theta) \propto \det(I_y(\theta))^{1/2}$ in terms of the hierarchy. Namely, we prove the following:

\begin{teo}\label{teoJeffreysrule} Under the hypotheses of Theorem \ref{FIM-hierarchical}, the usual Jeffreys rule priors satisfy the inequalities as follows. If the par\^ameter $\theta$ is multivariate (dimension $\ge 2$), then \begin{equation}\label{posteriorFisher2} \det(I_{y,w}(\theta))^{\frac{1}{2}} \ge \det(I_y(\theta))^\frac{1}{2} + \det\big(E_{y}[I_w(\theta\mid y)]\big)^{\frac{1}{2}}. \end{equation}
\end{teo}
\begin{proof}

Assume the dimension of the parameter space
$\Theta = [\theta]$ is $n\ge 2$. The Minkowski determinant theorem (see for instance Theorem 4.1.8 pg 115 of \cite{Marcus64}) states, for any $n\times n$ symmetric and positive semi-definite matrices $A$ and $B$,
\begin{equation}\label{minkowski}
\det(A+B)^{\frac{1}{n}} \ge \det(A)^{\frac{1}{n}} + \det(B)^{\frac{1}{n}}.
\end{equation}
Fisher information matrices are symmetric and positive semi-definite. By \eqref{FIM-decomposition} and \eqref{minkowski},
\begin{align}\label{ineq_fisher}
\det(I_{y,w}(\theta))^{\frac{1}{n}} &= \det\big(I_y(\theta) + E_{y}[I_w(\theta\mid y)]\big)^{\frac{1}{n}} \nonumber \\&\ge \det(I_y(\theta))^\frac{1}{n} + \det(E_{y}[I_w(\theta\mid y)])^\frac{1}{n}.
\end{align}
And, since $\frac{n}{2}\ge 1$,
\begin{align*}
\det(I_{y,w}(\theta))^{\frac{1}{2}} &= \big(\det(I_y(\theta))^\frac{1}{n} + \det(E_{y}[I_w(\theta\mid y)])^\frac{1}{n}\big)^{\frac{n}{2}}\\ &\ge \det(I_y(\theta))^\frac{1}{2} + \det(E_{y}[I_w(\theta\mid y)])^\frac{1}{2}.     
\end{align*}
Theorem \ref{teoJeffreysrule} is proved.
\end{proof}

The examples below will show, in many other situations, how to conveniently decompose $w$ in order to compute the Fisher information matrix of hierarchical models.

\section{Mixture model with known components}

Mixture models provide a useful representation of several applications such as gene expression, classification, etc. A simple mixture model with known components \citep{Gir88} may be written as
\begin{equation}\label{eqMix}
f(y\mid\theta)=\sum_{k=0}^p\theta_k\; g_k(y),
\end{equation}
with $S^n = \{\theta = [\theta_0,\ldots,\theta_p]^T \mid \theta_k>0 \mbox{ and }\sum_k \theta_k=1\}$ a simplex defining the weights of the mixture and $g_k(y)$, $k=0,\ldots,p$ completely specified density functions with known parameters. We set $w$ assuming values in $\{0,\ldots,p\}$ with categorical distribution $f_2(w=k\mid  \theta)=\theta_k$, with $k\in \{0,\ldots,p\}$. 
For each $w\in \{0,\ldots,p\}$, we fix a probability distribution $f_1(y\mid w) = g_w(y)$. The two-level hierarchical model $f_1(y\mid w)$ and $f_2(w\mid \theta)$ yields the mixture of probabilities $f(y\mid \theta) = \sum_w \theta_w g_w(y)$. The Fisher information matrix of the mixture model $f(y\mid \theta)$ is 
\begin{equation}\label{eqImix}
I_y(\theta)_{ij} = \int (g_i(y) - g_0(y))(g_j(y) - g_0(y))\frac{1}{f(y\mid \theta)}dy, 
\end{equation}
with $i,j=1,\ldots,n$. It is not a simple task to show the properness of Jeffreys' prior $\pi_y(\theta)$ derived from (\ref{eqImix}). In fact, this result was obtained by \citet{Gir88} and \citet{Robert18}. Corollary \ref{FIM-hierarchical_cor} gives an alternative proof of this result as we discuss as follows.

The Fisher information matrix based on the model $f_2(w\mid \theta)$ denoted by $I_w(\theta) = [I_w(\theta)_{ij}]$, with $i,j=1,\ldots,p$, is given by
\begin{equation*}
I_w(\theta)_{ij} = \frac{1}{\theta_0} + \frac{1}{\theta_i}\delta_{ij},
\end{equation*}
where $\delta_{ij}$ is the Kronecker delta. It is well known that the Jeffreys' prior of the categorical model $f_2$ is proper with normalizing constant $\int \det(I_w(\theta))^{\frac{1}{2}}d\theta = \frac{1}{2^{p+1}}\mathrm{vol(S^p_2)} = \frac{\pi^{\frac{p+1}{2}}}{\Gamma(\frac{p+1}{2})}$, where $S^p_2\subset \real^{p+1}$ denotes the round sphere in $\real^{p+1}$ of radius $2$. 
By Corollary \ref{FIM-hierarchical_cor}, the Jeffreys' prior of the incomplete model based on $I_y(\theta)$ is also proper. 

In the simpler model $w\mid \theta \sim Bern(\theta)$ resulting in $I_w(\theta)=\theta^{-1}(1-\theta)^{-1}$. Firstly, we take advantage of the hierarchy to obtain conjugated posterior for $w\mid y,\theta$. The complete conditional posterior distribution of $w$ is 
$$w\mid y,\theta \sim Bern \left ( \frac{\theta g_1(y)}{\theta g_1(y)+(1-\theta)g_2(y)}\right ).$$
The Fisher information of $w\mid \theta, y$ about $\theta$ is given by
$$I_w(\theta\mid y)=\frac{ g_1(y)/\theta}{\theta g_1(y)+(1-\theta)g_2(y)}+\frac{ g_2(y)/(1-\theta)}{\theta g_1(y)+(1-\theta)g_2(y)}+\frac{ (g_1(y)-g_2(y))^2}{[\theta g_1(y)+(1-\theta)g_2(y)]^2}.$$
Expectation of $I_w(\theta\mid y)$ with respect to the model $f(y\mid \theta)$ is obtained computationaly in the estimation algorithm.  
\begin{equation}\label{JefMix}
\pi(\theta)=\left\{ \theta^{-1}(1-\theta)^{-1} - E[I_w(\theta\mid y)]\right \}^{1/2}.
\end{equation}
The resulting prior must be integrable as $I_w(\theta)^{1/2}=\theta^{-1/2}(1-\theta)^{-1/2}$ also is. This result is immediate from Corollary \ref{FIM-hierarchical_cor}.

\section{Two level Hierarchical Gaussian model}

The usual random effect model depends strongly on the estimation of the unknown variances in each level of the hierarchy. 
The signal to noise ratio has been studied in several areas of applications.  
As follows we consider the two level hierarchical random effect model in its simplest version as a prototypical example.
\begin{eqnarray*}
{y}_i &=& {\omega}_i + {\epsilon_{1i}}\\
{\omega}_i & = &  \mu + {\epsilon}_{2i}
\end{eqnarray*}
with $ {\epsilon_{1i}} \sim N[  {0}, \delta^{-1}] $, $ {\epsilon}_{2i} \sim N[ {0}, \phi^{-1}] $ and $(\delta^{-1},\phi^{-1}) \in  \real^+ \times \real^+ $. Assume that $\delta$  is  known.
In this setting $
y\mid  \omega, \delta \sim  N[  {\omega}, \delta^{-1} ]$,
$\omega  | \mu \sim N[ \mu, \phi^{-1}   ]$ and integrating $\omega$ out yields the marginal model 
$ {y } | \phi  \sim N[  \mu , \delta^{-1}+\phi^{-1} ]$. This model may be seen as the simplest version of several hierarchical models which depends on random effect estimation. It is worth noting that  in this example $y \perp  \phi | \omega $, so the second term in the left hand side of the information identity  is null. 

The prior and posterior densities of $\omega$ are both Gaussian distributions with the same general form
$$
 \omega | \phi, y \sim N[a(\phi), A(\phi)]
 $$  with $a(\phi)=0$ and $A(\phi)=\frac{1}{\phi}$ in the prior distribution and $a(\phi)=  \frac{y}{\phi+1}$ and $A(\phi)=\frac{1}{\phi + 1}$ in the posterior distribution, where,  without loss of generality, we suppose that $\mu=0$ and $\delta=1$. The log conditional density  is  
 $$
 l(\phi)= -\frac{1}{2}log(2\pi)-\frac{1}{2} log(A(\phi) )- \frac{1}{2 A} (\omega -a(\phi))^2
 $$
 
 Twice differentiation and expectation in the distribution of $\omega$ results in
 $$E(-l''(\phi))=\frac{1}{2}\left ( \frac{A'(\phi)}{A(\phi)} \right )^2+\frac{(a(\phi))^2}{A(\phi)}$$
 In the prior distribution $A'(\phi)=-\frac{1}{\phi^2}$ and it follows that
$
I_{\omega}(\phi)=\frac{1}{2\phi^2}. 
$
In the posterior distribution $a'(\phi)=-\frac{y}{(1+\phi)^2}$, $A'(\phi)=-\frac{1}{(1+\phi)^2}$ and $
I_{\omega}(\phi\mid y)=\frac{1}{2(1+\phi)^2} + \frac{y^2}{(1+\phi)^3}.
$
Taking expectation in the marginal data distribution leads to
\begin{eqnarray} 
 E_y[I_{\omega}(\phi | y) ] &=&    \frac{ 2+\phi}{2\phi(\phi+1)^2}
   \end{eqnarray}

Finally, applying the proposed Fisher decomposition
\begin{eqnarray}\label{final2levels}
I_y(\phi)= I_w(\phi)-E[I_w(\phi\mid y)] = \frac{1}{2} \left [ \frac{1}{\phi^2} - \frac{\phi+2}{\phi(\phi+1)^2}\right ]
\end{eqnarray}
In this example, the marginal model is easily obtained and usual differentiation of the log likelihood leads to the same result obtained in (\ref{final2levels}). However, for more general random effect models marginalization might be not feasible and this procedure would still apply to obtain the required Fisher information.

\section{Student-t model with unknown degrees of freedom}

Let $y\mid \theta \sim {\rm ST}(\theta)$ be the standard Student-t model with fixed mean and precision and unknown degrees of freedom $\nu$. This model robustfy the Gaussian model by allowing for fatter tails. Several authors have dealt with reference prior specification for the tail parameter $\nu$ \citep{Fons08,Wal14,Rue17}. This model may be rewritten in an hierarchical setting as
\begin{align*}
y\mid w \sim N(0,1/w),\\
w\mid \theta \sim Ga(\theta/2,\theta/2).
\end{align*}
This model has two levels of hierarchy with $\theta$ appearing in the second level only, as in \eqref{hierarchi_particular}. In this case, by Corollary \ref{FIM-hierarchical_cor}, the Fisher decomposition simplifies to
\begin{equation}\label{eqT}
I_y(\theta)=I_w(\theta)-E_y[I_w(\theta\mid y)].
\end{equation}
The first term is obtained from the latent variable prior distribution $f_2(w\mid \theta)$ which is a gamma distribution and twice differentiation and expectation with respect to $w\mid \theta$ leads to 
\begin{equation}\label{Decomp1}
I_{w}(\theta)=\frac{1}{4}\;\psi^{(2)}\left (\frac{\theta}{2}\right )-\frac{1}{2\theta}.
\end{equation}
The second part in (\ref{eqT}) is computed based on the complete condicional posterior distribution of the latent variable $w$ which is $w\mid \theta,y\sim Ga((\theta+1)/2,(y^2+\theta)/2)$. Twice differentiation and taking expectation in the distribution of $w\mid \theta,y$ results in
$$
I_w(\theta\mid y)=\frac{1}{4}\;\psi^{(2)}\left (\frac{\theta+1}{2}\right )+\frac{\theta+1}{2(\theta+y^2)^2}-\frac{1}{\theta+y^2}.
$$
Now we take expectations with respect to the model $f(y\mid \theta)$. For this model the result is analitic and given by
\begin{equation}\label{Decomp2}
E[I_w(\theta\mid y)]=\frac{1}{4}\;\psi^{(2)}\left (\frac{\theta+1}{2}\right )+\frac{\theta+2}{2\theta(\theta+3)}-\frac{1}{\theta+1}.
\end{equation}
Equations (\ref{Decomp1}) and ($\ref{Decomp2}$) together result in the usual Jeffreys rule prior for the degrees of freedom of the Student-t model. The Jeffreys prior is bounded above by the function $p(\theta)=k I_w(\theta)^{1/2}$ which is not integrable. However, $I_y(\theta)\asymp \mathcal{O} (\theta^{-4})$ as $\theta \to \infty$. Thus, the Jeffreys prior for the degree of freedom of the Student-t model is proper.

\section{Hierarchical priors for variable selection in regression analysis}

This example illustrates the use of hierarchical models for variable selection in regression analysis and a convenient use of our proposed Fisher decomposition. 
Consider the usual linear regression model for predictor variables $X$ and responses $y$
\begin{equation}\label{eqReg}
y\mid w_1\sim N(X w_1,\sigma^2 I_n ),
\end{equation}
where $y$ is the $n\times 1$ vector of responses, $X$ is the $n \times p$ matrix of covariates, and $w_1$ is the $n\times 1$ vector of regression coefficients. Consider the problem of variable selection in that context, that is, if $p$ is large it is desirable to find some $w_1$'s equal or close to zero. \citet{Tib96} proposed a method called lasso, least absolute shrinkage selection operator, which is able to produce coefficients exactly equal to zero in regression models. The lasso estimate $\hat w_1$ is defined by
\begin{equation}
\hat w_1=\arg\min\left\{ \sum_{i=1}^{n}\left ( y_i-\sum_{j=1}^{p}w_{1,j}x_{ij}\right )^2 \right \} \mbox{ subject to }\;\sum_{j=1}^{p}|w_{1,j}|\leq t.
\end{equation}
The parameter $t$, the Lasso parameter, controls the amount of shrinkage that is applied to the estimates.  Let $\hat w_{1,j}^0$ be the full least square estimates and let $t_0=\sum_{j=1}^{p}|\hat w_{1,j}^0|$. Values of $t<t_0$ will cause shrinkage of the solutions towards zero. This method aims to improve prediction accuracy and to be more interpretable, as we may focus only in the strongest effects. However, estimation of $t$ is not an easy task, \citet{Tib96} comments that since the lasso estimate is a non-linear and non-differentiable function of the response, it is difficult to obtain an accurate estimate of its standard error.

The lasso constraint $\sum_{j=1}^p |w_{1,j}|\leq t$ is equivalent to the addition of a penalty term $\gamma\sum |w_{1,j}|$ to the residual sum of squares, that is, we should minimize
\begin{equation}
\sum_{i=1}^{n}\left ( y_i-\sum_{j=1}^{p}w_{1,j}x_{ij}\right )^2 + \gamma\sum_{j=1}^p |\beta_j|.
\end{equation}
In the Bayesian context, this is equivalent to the use of independent Laplace priors for the regression coefficients \citep{Park08}, that is,
\begin{equation}\label{LapPrior}
p(w_{1,j})=\frac{\gamma}{2}\exp\left \{-\gamma|w_{1,j}|\right \}.
\end{equation}
We may derive the Lasso estimate as the posterior mode under the prior (\ref{LapPrior}) for the $w_1$`s. The choice of the Bayesian Lasso parameter is also not trivial and several methods has been proposed to estimate this parameter such as cross validation and empirical Bayes, however, these methods are often unstable. It is pointed out by \citet{Park08} that the standard error obtained for the lasso parameter are not fully satisfactory. 

\cite{Mal14} rewrites the Laplace prior for $w_1$ in a convenient hierarchical way which allows for analytical full conditional distributions of the latent variables and model parameters. The Lasso prior is obtained as an uniform scale mixture by considering the conditional setting
\begin{eqnarray}
w_{1,j}\mid w_{2,j} & \sim & Unif(-\sigma w_{2,j},\sigma w_{2,j})\\
w_{2,j} \mid \theta & \sim & Ga (2,\theta)
\end{eqnarray}
As follows we obtain the Jeffreys prior for $\theta$ based on the Fisher decomposition \eqref{hierarchi_particular} which simplifies to
\begin{equation}
I_y(\theta)=I_{w_2}(\theta)-E_y[I_w(\theta\mid y)],
\end{equation}
with $w=(w_1,w_2)$. The Fisher information for the bottom level in the hierarchy is easily obtained as $I_{w_2}(\theta)=\frac{2p}{\theta^2}$. In order to obtain $E_y[I_w(\theta\mid y)]$ we take advantage of the known complete conditional distribution of $w\mid \theta,y$ which is presented in \cite{Mal14} as 
\begin{eqnarray}
w_{1,j}\mid w_{2,j},\theta,y & \sim & N(\hat w_{1}^{OLS},\sigma^2 (X' X)^{-1})\prod_{j=1}^pI(|w_{1,j}|\leq \sigma w_{2,j})\\
w_{2,j} \mid \theta,y & \sim & Exp (\theta),
\end{eqnarray}
with $\hat w_{1}^{OLS}$ the ordinary least squared estimator of $w_1$ in the usual regression model \eqref{eqReg}. The distribution of $w_{1,j}\mid w_{2,j},\theta,y$ does not depend on $\theta$. Thus, $E_y[I_w(\theta\mid y)]$ depend on the distribution of $w_{2,j} \mid \theta,y$ only and is trivialy obtained as $E_y[I_{w}(\theta\mid y)]=\frac{p}{\theta^2}$. The final Jeffreys prior has closed form given by
\begin{equation}
    \pi(\theta)\propto \theta^{-1}.
\end{equation}
Notice that this example illustrates the usefulness of our proposed decomposition in other robust regression models such as the ridge regression, elastic net and horse shoe model.

\section{Hyperbolic model}

This example illustrates the use of our proposed decomposition when all levels of hierarchy depend on the parameter of interest. In particular, we present results for the Hyperbolic model which is an extension of the Gaussian model allowing for asymmetric behavior. The Jeffreys rule prior for this model has been proposed by \cite{Fons12}, however, the prior is only obtained computationally and no propertness of the resulting proposal is proved. 

A random variable is said to have a Hyperbolic distribution with parameters $\alpha$, $\beta$, $\delta$ e $\mu$ when its density is given by
\begin{equation} \label{hyp}
g(y;\alpha,\beta,\delta,\mu)=\frac{\sqrt{\alpha^2-\beta^2}}{2\alpha\delta K_1(\delta\sqrt{\alpha^2-\beta^2})}\exp \left\{ -\alpha\sqrt{\delta^2+(y-\mu)^2}+\beta(y-\mu) \right\}
\end{equation}
with $\;y,\mu\in \Re,\;\delta>0\;e\;|\beta|<\alpha$ and $K_1(.)$ the modified Bessel function of the tird kind with index 1. This model can be written alternatively as a mixture of a Gaussian and a GIG\footnote{If $W \sim GIG(\rho ,\gamma,\kappa)$ then its density is given by
$$c(\rho,\gamma,\kappa)\; w^{\rho -1}\exp\left \{  -\frac{1}{2}(\kappa^2 w^{-1}+\gamma^2 w) \right \},\;\kappa,\gamma\geq 0,\;\; \rho  \in \Re, \;\; w>0,$$
with $c(\rho,\gamma,\kappa)=\frac{(\gamma/\kappa)^{\rho}}{2K_{\rho}(\gamma\kappa)}$. J{\o}rgensen (1982) presents details about this class.} distribution, that is,
\begin{equation} \label{mix}
f(y|\lambda,\alpha,\beta,\delta,\mu)=\int_{0}^{\infty}f_1(y|\mu,\beta,w) f_2(w|\lambda,\delta,\alpha,\beta) dw,
\end{equation}
with $f_1$ the Gaussian distribution with mean $\mu+\beta w$ and variance $w$ and $f_2$ the Generalized Inverse Gaussian distribution (GIG) with parameters 1, $\sqrt{\alpha^2-\beta^2}$ and $\delta$.

In particular, consider the standard Hyperbolic model (\ref{hyp}) for $\mu=0$, $\delta=1$ and $\alpha^2=2\beta^2$ which corresponds to an asymmetric density function for $y$. Let $y\mid \theta \sim {\rm Hyp}(\sqrt{2\theta^2},\theta,1,0)$. This model may be rewritten as
\begin{align*}\label{modHyp1}
y|w,\theta \sim N(\theta w,w),\\
w\mid \theta \sim GIG(1,\theta,1).
\end{align*}

Now we illustrate how to compute the Jeffreys prior using the proposed decomposition when both levels of hierarchy depend on $\theta$. By Theorem \ref{FIM-hierarchical}, the Fisher decomposition is given by
\begin{equation}\label{eqH}
I_y(\theta)=I_w(\theta)-E[I_w(\theta\mid y)] +E[I_y(\theta\mid w)].    
\end{equation}

In this case the latent variable $w$ has a GIG prior distribution and differentiating $\log f_2(w\mid \theta)$ and taking expectations with respect to $w\mid \theta$ leads to
\begin{equation}\label{DecompHyp}
I_{w}(\theta)=S_1(\theta),
\end{equation}
with $S_1(\theta)=\frac{K_{3}(\theta)}{K_1(\theta)}-R_1^2(\theta)$, and $R_1(\theta)=\frac{K_{2}(\theta)}{K_1(\theta)}$. Notice that in this example the first level of hierarchy $f_1(y\mid w, \theta)$ also depends on $\theta$ thus, to obtain the upper bound for the Jeffreys prior, we need to compute $E[I_y(\theta\mid w)]$. In the first level $y$ has a Gaussian distribution, thus differentiating $\log f_1(y\mid w,\theta)$ and taking expectations with respect to $y\mid w,\theta$ leads to
\begin{equation}\label{DecompHyp2}
E[I_{y}(\theta\mid w)]=\frac{1}{\theta}R_1(\theta),
\end{equation}
The resulting upper bound for the Jeffreys rule prior is given by the function
$$\pi^*(\theta)=k \left \{S_1(\theta)+\frac{1}{\theta}R_1(\theta)\right \}^{1/2}.$$
This prior is not integrable and the propertness of the prior distribution must be investigated using the resulting Jeffreys prior. The second term in the right hand side of (\ref{eqH}) is computed based on the complete condicional posterior distribution of the latent variable $w$ which in this model is also a GIG distribution given by
$$w\mid \theta,y\sim GIG(1/2,\sqrt{2\theta^2},\sqrt{y^2+1}).$$ 
Twice differentiation  and taking expectation in the distribution of $w\mid \theta,y$ results in
$$
I_w(\theta\mid y)=2(y^2+1)S_{1/2}(\theta\sqrt{2(y^2+1)}).
$$
Now we take expectations with respect to the model $f(y\mid \theta)$ to obtain the Jeffreys rule prior for the asymmetry parameter in the Hyperbolic model $y\mid \theta \sim {\rm Hyp}(\sqrt{2\theta^2},\theta,1,0)$. 
\begin{equation}\label{JefHyp}
\pi(\theta)=\left\{ S_1(\theta)+\frac{1}{\theta}R_1(\theta) - E[I_w(\theta\mid y)]\right \}^{1/2}.
\end{equation}
The resulting prior has a term that must be computed numerically. Notice that instead of computing $\pi(\theta)$ directly by numerical integration a part of it is obtained analytically which leads to reduction of variance. 

\section{Conclusions and further developments}

This paper presented a convenient Fisher decomposition which facilitates Jeffreys prior computation even when model marginalization is not available analytically. We have illustrated our proposal with well known examples from the literature which either present ill behaved likelihoods or have hyperparameters which are difficult to set informative priors. Many other hierarchical models could be exploit and properties from the resulting prior could be investigated using the proposed Fisher decomposition. In particular, the variable selection models with other mixing distributions other than the Lasso prior could be studied and properties of propertiness of the resulting Jeffreys priors could be proved from our theoretical results. 
\bibliography{ref}
\bibliographystyle{jasa}
\end{document}